\documentclass[11pt]{amsart}
\usepackage{graphicx}
\usepackage[mathcal]{euscript}
\usepackage{float}
\usepackage{cite}
\usepackage{verbatim}

\restylefloat{figure}

\newtheorem{theorem}{Theorem}[section]

\newtheorem{corollary}[theorem]{Corollary}
\newtheorem{proposition}[theorem]{Proposition}

\theoremstyle{definition}

\newtheorem{question}{Question}

\theoremstyle{remark}

\numberwithin{equation}{section}

\def\Span{\operatorname{span}}
\def\alt{\operatorname{alt}}
\def\dalt{\operatorname{dalt}}
\def\warp{\operatorname{warp}}
\def\OS{Ozsv\'ath-Szab\'o}

\usepackage{tikz}

\usetikzlibrary{arrows,shapes,snakes,backgrounds,patterns}

\pgfdeclarelayer{background}
\pgfdeclarelayer{background2}
\pgfdeclarelayer{background2a}
\pgfdeclarelayer{background2b}
\pgfdeclarelayer{background3}
\pgfdeclarelayer{background4}
\pgfdeclarelayer{background5}
\pgfdeclarelayer{background6}
\pgfdeclarelayer{background7}

\pgfsetlayers{background7,background6,background5,background4,background3,background2b,background2a,background2,background,main}

\begin{document}

\title{Alternating distances of knots and links}

\author{Adam M. Lowrance}
\address{Department of Mathematics\\
Vassar College\\
Poughkeepsie, NY} 
\email{adlowrance@vassar.edu}
\thanks{The author was supported by an AMS-Simons Travel Grant.}

\subjclass{}
\date{}

\begin{abstract}
An alternating distance is a link invariant that measures how far away a link is from alternating. We study several alternating distances and demonstrate that there exist families of links for which the difference between certain alternating distances is arbitrarily large. We also show that two alternating distances, the alternation number and the alternating genus, are not comparable.
\end{abstract}

\maketitle

\section{Introduction}

Alternating links play an important role in knot theory and $3$-manifold geometry and  topology. Link invariants are often easier to compute and take on special forms for alternating links. Moreover, the complements of alternating links have interesting topological and geometric structures. Many generalizations of alternating links exist, and a particular generalization can give rise to an invariant that measures how far a link is from alternating. We study several such invariants, which we call alternating distances. 

A link $L$ is {\em split} if it has a separating sphere, i.e. a two-sphere $S^2$ in $S^3$ such that $L$ and $S^2$ are disjoint and each component of $S^3-S^2$ contains at least one component of $L$. We will mostly be concerned with non-split links, that is links with no separating spheres.
A real valued link invariant $d(L)$ is an {\em alternating distance} if it satisfies the following conditions.
\begin{enumerate}
\item For any non-split link, $d(L)\geq 0$.
\item For any non-split link, $d(L)=0$ if and only if $L$ is alternating.
\item If $L_1$ and $L_2$ are non-split links and $L_1\#L_2$ is any connected sum of $L_1$ and $L_2$, then $d(L_1\# L_2) \leq d(L_1) + d(L_2)$.
\end{enumerate}
The connected sum $L_1\#L_2$ depends on a choice of components in $L_1$ and $L_2$. However, condition (3) above must be true for any choice of connected sum. We will frequently use the notation $L_1\# L_2$ to denote an arbitrary choice of connected sum.

We consider the following invariants. The dealternating number, denoted $\dalt(L)$, and the alternation number, denoted $\alt(L)$, are defined by counting crossing changes. The relationship between the minimum crossing number $c(L)$ of a link $L$ and the span of the Jones polynomial $V_L(t)$ of $L$ was used to prove some of Tait's famous conjectures, and we study the difference $c(L)-\Span V_L(t)$. The Turaev genus, denoted $g_T(L)$, and the alternating genus, denoted $g_{\alt}(L)$, are the genera of certain surfaces associated to $L$. The warping span, denoted $\warp(K)$, is defined by examining the over-under behavior as one travels along the knot or link. Precise definitions of these invariants are given in Section \ref{section:definitions}.

Let $d_1$ and $d_2$ be real valued link invariants and let $\mathcal{F}$ be a family of links. We say $d_2$ {\em dominates} $d_1$ on $\mathcal{F}$, and write $d_1(\mathcal{F})\ll d_2(\mathcal{F})$, if for each positive integer $n$, there exists a link $L_n\in\mathcal{F}$ such that $d_2(L_n) - d_1(L_n)
\geq n$. 

In Section \ref{section:examples}, we examine three families of links. The first family $\mathcal{F}(W_n)$ consists of iterated Whitehead doubles of the figure-eight knot. The second family $\mathcal{F} (\widetilde{T}(p,q))$ consists of links obtained by changing certain crossings of torus links. The third family $\mathcal{F}(T(3,q))$ consists of the $(3,q)$-torus knots. 
\begin{theorem} 
\label{theorem:main1}
Let $\mathcal{F}(W_n)$, $\mathcal{F}(\widetilde{T}(p,q))$, and $\mathcal{F}(T(3,q))$ be the families of links above.
\begin{enumerate}
\item The dealternating number, $c(L)-\Span V_L(t)$, and Turaev genus dominate the alternation number on $\mathcal{F}(W_n)$.
\item The dealternating number, Turaev genus, $c(L) - \Span V_L(t)$, and alternation number dominate the alternating genus on $\mathcal{F}(\widetilde{T}(p,q))$.
\item The dealternating number, alternation number, $c(L)-\Span V_L(t)$, and Turaev genus dominate the warping span on $\mathcal{F}(T(3,q))$.
\item The difference $c(L)-\Span V_L(t)$ dominates the dealternating number, alternation number, alternating genus, and Turaev genus on $\mathcal{F}(T(3,q))$.
\end{enumerate}
\end{theorem}

Two real valued link invariants $d_1$ and $d_2$ are said to be {\em comparable} if either $d_1(L)\leq d_2(L)$ or $d_2(L)\leq d_1(L)$ for all links $L$. The invariants $d_1$ and $d_2$ are not comparable if there exists links $L$ and $L'$ such that $d_1(L)<d_2(L)$ and $d_1(L')>d_2(L')$.
\begin{theorem}
\label{theorem:main2}
The alternation number and alternating genus of a link are not comparable.
\end{theorem} 

This paper is organized as follows. In Section \ref{section:definitions}, we define the invariants mentioned in Theorems \ref{theorem:main1} and \ref{theorem:main2}. In Section \ref{section:bounds}, we describe some lower bounds for the invariants. In Section \ref{section:examples}, we define several families of links and use them to prove Theorems \ref{theorem:main1} and \ref{theorem:main2}. In Section \ref{section:questions}, we discuss some open questions about our invariants.

{\bf Acknowledgement:} The author is grateful for comments from Oliver Dasbach, Radmila Sazdanovi\'c, and Alexander Zupan.

\section{The invariants}
\label{section:definitions}
In this section, the invariants of Theorems \ref{theorem:main1} and \ref{theorem:main2} are defined. We show that each one is an alternating distance, and discuss some known relationships between them.

\subsection{Dealternating number}

Adams et al. \@ \cite{Adams:Almost} define {\em almost alternating links} to be non-alternating links with a diagram such that one crossing change makes the diagram alternating. They use the notion of number of crossing changes needed to make a diagram alternating to define the dealternating number of a link. The {\em dealternating number} of a link diagram $D$, denoted $\dalt(D)$, is the minimum number of crossing changes necessary to transform $D$ into an alternating diagram. The {\em dealternating number} of a link $L$, denoted $\dalt(L)$, is the minimum dealternating number of any diagram of $L$. A link with dealternating number $k$ is also called $k$-almost alternating.

\begin{proposition}
The dealternating number of a link is an alternating distance.
\end{proposition}
\begin{proof}
The definition of the dealternating number implies that it is always non-negative and equals zero if and only if the link is alternating. Suppose that $L_1$ and $L_2$ are links with diagrams $D_1$ and $D_2$ such that $\dalt(D_1)=\dalt(L_1)$ and $\dalt(D_2)=\dalt(L_2)$. For any choice of connected sum $L_1\#L_2$, there exists some choice $D_1\# D_2$ of connected sum of diagrams $D_1$ and $D_2$ such that $D_1\# D_2$ is a diagram of $L_1\# L_2$ and $\dalt(D_1\#D_2)=\dalt(D_1)+\dalt(D_2)=\dalt(L_1)+\dalt(L_2)$. Thus $\dalt(L_1\#L_2)\leq \dalt(L_1)+\dalt(L_2)$ and hence the dealternating number of a link is an alternating distance.
\end{proof}

\subsection{Alternation number}
Kawauchi \cite{Kawauchi:Alternation} uses crossing changes in a slightly different manner to define the alternation number of a link. The {\em alternation number} of a link diagram $D$, denoted alt$(D)$, is the minimum number of crossing changes necessary to transform $D$ into some (possibly non-alternating) diagram of an alternating link. The {\em alternation number} of a link $L$, denoted alt$(L)$, is the minimum alternation number of any diagram of $L$. The alternation number of $L$ is also the Gordian distance from $L$ to the set of alternating links \cite{Murakami:Metrics}. It is immediate from their definitions that 
\begin{equation}
\label{eq::altdalt}
\alt(L)\leq\dalt(L)
\end{equation}
for any link $L$.

\begin{proposition}
The alternation number of a link is an alternating distance.
\end{proposition}
\begin{proof}
The definition of the alternation number implies that it is always non-negative and equals zero if and only if the link is alternating. Suppose that $L_1$ and $L_2$ are links with diagrams $D_1$ and $D_2$ respectively such that $\alt(D_1)=\alt(L_1)$ and $\alt(D_2)=\alt(L_2)$. Let $\widetilde{D}_1$ and $\widetilde{D}_2$ be the diagrams of alternating links obtained from $D_1$ and $D_2$ respectively via the minimum number of crossing changes. Let $L_1\#L_2$ be a connected sum of $L_1$ and $L_2$, and let $D_1\# D_2$ be a connected sum of $D_1$ and $D_2$ such that $D_1\#D_2$ is a diagram of $L_1\#L_2$. Then $D_1\#D_2$ can be transformed into $\widetilde{D}_1\#\widetilde{D}_2$ (an alternating link) via $\alt(D_1)+\alt(D_2)=\alt(L_1)+\alt(L_2)$ crossing changes. Hence $\alt(L_1\#L_2)\leq \alt(L_1)+\alt(L_2)$, and thus the alternation number is an alternating distance.
\end{proof}

\subsection{Crossing number and the span of the Jones polynomial}
In the late 19\textsuperscript{th} century, Tait \cite{Tait:Conjecture} conjectured that a certain type of alternating link diagram (called reduced) has minimal crossing number among all diagrams for that link. This conjecture remained undecided until the discovery of the Jones polynomial \cite{Jones1}, when combined work of Kauffman \cite{Kauffman2} and Murasugi \cite{Murasugi} proved the conjecture to be true. 

Let $V_L(t)$ denote the Jones polynomial of $L$, and let $\max \deg V_L(t)$ and $\min \deg V_L(t)$ denote the maximum and minimum power of $t$ in $V_L(t)$ with non-zero coefficient respectively. Define $\Span V_L(t) = \max \deg V_L(t) - \min \deg V_L(t)$. Let $c(L)$ denote the minimum number of crossings in any diagram of the link $L$. 
\begin{proposition}
The difference $c(L)-\Span V_L(t)$ is an alternating distance.
\end{proposition}
\begin{proof}
Murasugi \cite{Murasugi} and Kaufman \cite{Kauffman2} prove that $\Span V_L(t) \leq c(L)$ for any non-split link $L$. Furthermore, when $L$ is non-split, Murasugi proves that $\Span V_L(t) = c(L)$ if and only if $L$ is a connected sum of alternating links, which happens if and only if $L$ is alternating. The span of the Jones polynomial is additive under connected sums, i.e. $\Span V_{L_1\#L_2}(t) = \Span V_{L_1}(t) + \Span V_{L_2}(t)$. It is a long-standing open question whether crossing number is additive, but it is easy to see that crossing number is sub-additive, i.e. $c(L_1\# L_2) \leq c(L_1) + c(L_2)$.
Therefore $c(L_1\#L_2) - \Span V_{L_1\#L_2}(t) \leq (c(L_1) - \Span V_{L_1}(t)) + (c(L_2)-\Span V_{L_2}(t)).$ Hence $c(L)-\Span V_L(t)$ is an alternating distance.
\end{proof}

\subsection{Turaev genus}
Turaev \cite{Turaev:SimpleProof} gives a simplified proof of Tait's conjecture mentioned above, where he associates to each link diagram $D$ a surface $F(D)$, now known as the {\em Turaev surface of $D$}. Each crossing of $D$ can be resolved in either an $A$-resolution or a $B$-resolution, as depicted in Figure \ref{figure:smoothing}.
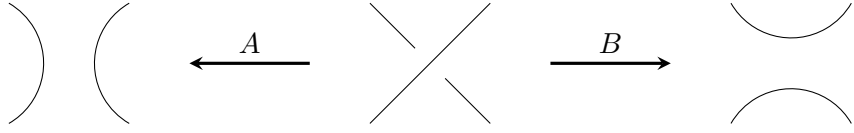
\begin{figure}[h]
$$\begin{tikzpicture}[>=stealth, scale=.8]
\draw (-1,-1) -- (1,1);
\draw (-1,1) -- (-.25,.25);
\draw (.25,-.25) -- (1,-1);
\draw (-3,0) node[above]{$A$};
\draw[->,very thick] (-2,0) -- (-4,0);
\draw (3,0) node[above]{$B$};
\draw[->,very thick] (2,0) -- (4,0);
\draw (-5,1) arc (120:240:1.1547cm);
\draw (-7,-1) arc (-60:60:1.1547cm);
\draw (5,1) arc (210:330:1.1547cm);
\draw (7,-1) arc (30:150:1.1547cm);
\end{tikzpicture}$$
\caption{The $A$ and $B$ resolutions of a crossing}
\label{figure:smoothing}
\end{figure}
A collection $s$ of simple closed curves obtained by choosing either an $A$-resolution or a $B$-resolution for each crossing of $D$ is a {\em state} of $D$. The state $s_A(D)$ obtained by choosing an $A$-resolution at each crossing is called the {\em all-$A$ state} of $D$, and similarly, the state $s_B(D)$ obtained by choosing a $B$-resolution at each crossing is called the {\em all-$B$ state} of $D$. 

Let $\Gamma$ be the $4$-valent graph obtained from $D$ by forgetting the ``over-under" information at each crossing. Regard $\Gamma$ as embedded in $S^2$, and thicken the sphere to $S^2\times[-1,1]$ such that $\Gamma$ is a subset of $S^2\times \{0\}$. We first describe the intersection of the Turaev surface $F(D)$ and $S^2\times[-1,1]$. Outside of neighborhoods of the vertices of $\Gamma$, the Turaev surface intersects $S^2\times[-1,1]$ in $\Gamma\times [-1,1]$. In a neighborhood of each crossing, the Turaev surface intersects $S^2\times[-1,1]$ in a saddle positioned so that $F(D)\cap S^2\times\{-1\}$ is the all-$B$ state $s_B(D)$ and $F(D)\cap S^2\times\{1\}$ is the all-$A$ state $s_A(D)$, as depicted in Figure \ref{figure:saddle}. 
\begin{figure}[h]
$$\begin{tikzpicture}
\begin{scope}[thick]
\draw [rounded corners = 10mm] (0,0) -- (3,1.5) -- (6,0);
\draw (0,0) -- (0,1);
\draw (6,0) -- (6,1);
\draw [rounded corners = 5mm] (0,1) -- (2.5, 2.25) -- (0.5, 3.25);
\draw [rounded corners = 5mm] (6,1) -- (3.5, 2.25) -- (5.5,3.25);
\draw [rounded corners = 5mm] (0,.5) -- (3,2) -- (6,.5);
\draw [rounded corners = 7mm] (2.23, 2.3) -- (3,1.6) -- (3.77,2.3);
\draw (0.5,3.25) -- (0.5, 2.25);
\draw (5.5,3.25) -- (5.5, 2.25);
\end{scope}

\begin{pgfonlayer}{background2}
\fill [blue!30!white]  [rounded corners = 10 mm] (0,0) -- (3,1.5) -- (6,0) -- (6,1) -- (3,2) -- (0,1); 
\fill [blue!30!white] (6,0) -- (6,1) -- (3.9,2.05) -- (4,1);
\fill [blue!30!white] (0,0) -- (0,1) -- (2.1,2.05) -- (2,1);
\fill [blue!30!white] (2.23,2.28) --(3.77,2.28) -- (3.77,1.5) -- (2.23,1.5);

\fill [white, rounded corners = 7mm] (2.23,2.3) -- (3,1.6) -- (3.77,2.3);
\fill [blue!30!white] (2,2) -- (2.3,2.21) -- (2.2, 1.5) -- (2,1.5);
\fill [blue!30!white] (4,2) -- (3.7, 2.21) -- (3.8,1.5) -- (4,1.5);
\end{pgfonlayer}

\begin{pgfonlayer}{background4}
\fill [blue!90!red] (.5,3.25) -- (.5,2.25) -- (3,1.25) -- (2.4,2.2);
\fill [rounded corners = 5mm, blue!90!red] (0.5,3.25) -- (2.5,2.25) -- (2,2);
\fill [blue!90!red] (5.5,3.25) -- (5.5,2.25) -- (3,1.25) -- (3.6,2.2);
\fill [rounded corners = 5mm, blue!90!red] (5.5, 3.25) -- (3.5,2.25) -- (4,2);
\end{pgfonlayer}

\draw [thick] (0.5,2.25) -- (1.6,1.81);
\draw [thick] (5.5,2.25) -- (4.4,1.81);
\draw [thick] (0.5,2.75) -- (2.1,2.08);
\draw [thick] (5.5,2.75) -- (3.9,2.08);

\begin{pgfonlayer}{background}
\draw [black!50!white, rounded corners = 8mm, thick] (0.5, 2.25) -- (3,1.25) -- (5.5,2.25);
\draw [black!50!white, rounded corners = 7mm, thick] (2.13,2.07) -- (3,1.7)  -- (3.87,2.07);
\end{pgfonlayer}
\draw [thick, dashed, rounded corners = 2mm] (3,1.85) -- (2.8,1.6) -- (2.8,1.24);
\draw (0,0.5) node[left]{$\Gamma$};
\draw (1.5,3) node{$s_A$};
\draw (4.5,3) node{$s_A$};
\draw (3.8,.8) node{$s_B$};
\draw (5.3, 1.85) node{$s_B$};
\end{tikzpicture}$$
\caption{In a neighborhood of each vertex of $\Gamma$ a saddle surface transitions between the all-$A$ and all-$B$ states.}
\label{figure:saddle}
\end{figure}
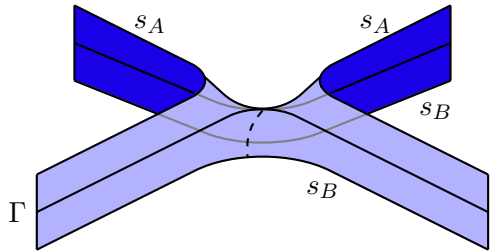
Summarizing, the intersection of the Turaev surface with $S^2\times[-1,1]$ is a cobordism between $s_A(D)$ and $s_B(D)$ whose saddle points correspond to the crossings of $D$. Outside of $S^2\times[-1,1]$ the Turaev surface $F(D)$ is a collection of disks capping off the components of $s_A(D)$ and $s_B(D)$. 

The Turaev surface $F(D)$ is always oriented, and if $D$ is a diagram of a non-split link $L$, then $F(D)$ is connected. The {\em Turaev genus $g_T(D)$} of a diagram $D$ is defined to be the genus of the Turaev surface $F(D)$. It can be shown that 
\begin{equation}
\label{equation:Turaevgenus}
g_T(D) = \frac{1}{2}(2 + c(D) -|s_A(D)| - |s_B(D)|),
\end{equation}
where $c(D)$ is the number of crossings of $D$, and $|s_A(D)|$ and $|s_B(D)|$ denote the number of components in the all-$A$ and all-$B$ states of $D$. The Turaev genus $g_T(L)$ of a non-split link $L$ is the minimum genus of the Turaev surface of $D$ where $D$ is any diagram of $L$, i.e.
$$g_T(L) = \min\{g_T(D)~|~D~\text{is a diagram of}~L\}.$$

A closed, oriented surface $\Sigma\subset S^3$ is a Heegaard surface if both components of $S^3-\Sigma$ are handlebodies. If $\Sigma$ is a Heegaard surface in $S^3$, then every link $L$ has an isotopy class representative that lies within a thickened neighborhood $\Sigma\times[-1,1]$ of $\Sigma$. Let $\pi:\Sigma\times[-1,1]\to\Sigma$ be projection onto the first factor. The link $L$ is said to have an {\em alternating projection} to $\Sigma$ if $\pi(L)$ has only transverse double points and the crossings of the projection $\pi(L)$ alternate as one travels along each component of $L$. 

Dasbach et al. \@ \cite{DFKLS:KauffmanDessins} prove the following facts about the Turaev surface $F(D)$ and Turaev genus of a non-split link.
\begin{itemize}
\item The Turaev surface $F(D)$ is a Heegaard surface in $S^3$.
\item The link $L$ has an alternating projection $\pi$ to $F(D)$.
\item The complement $F(D) - \pi(L)$ of the projection $\pi$ is a disjoint union of disks.
\item The Turaev surface of an alternating diagram is a sphere, and $g_T(L)=0$ if and only if $L$ is alternating.
\end{itemize}

\begin{proposition}
The Turaev genus of a link is an alternating distance.
\end{proposition}
\begin{proof}
Since Turaev genus is a minimum genus of a surface, it is always non-negative. As mentioned above, a link $L$ has Turaev genus zero if and only if it is alternating. If $D_1$ and $D_2$ are link diagrams and $D_1\# D_2$ is any connected sum, then $|s_A(D_1\#D_2)| = |s_A(D_1)| + |s_A(D_2)| -1$ and $|s_B(D_1\#D_2)| = |s_B(D_1)| + |s_B(D_2)| -1$. Then Equation \ref{equation:Turaevgenus} implies that $g_T(D_1\#D_2) = g_T(D_1)+g_T(D_2)$. Hence if $L_1$ and $L_2$ are links and $L_1\#L_2$ is any connected sum, then $g_T(L_1\#L_2)\leq g_T(L_1)+g_T(L_2).$ Therefore, the Turaev genus of a link is an alternating distance.
\end{proof}

Turaev \cite{Turaev:SimpleProof} shows that $g_T(D)\leq c(D) - \Span V_L(t)$ for any diagram $D$ of the link $L$. Minimizing over all diagrams of the link $L$, one obtains
\begin{equation}
\label{ineq:TuraevJones}
g_T(L) \leq c(L) - \Span V_L(t).
\end{equation}
Abe and Kishimoto \cite{Abe:Dealternating} examine the behavior of the Turaev surface under crossing changes to show that $g_T(D)\leq \dalt(D)$ for any link diagram $D$. Consequently, for any link $L$,
\begin{equation}
\label{ineq:TuraevDalt}
g_T(L)\leq \dalt(L).
\end{equation}
Champanerkar and Kofman's recent survey \cite{CK:TuraevSurvey} gives many open questions concerning the Turaev genus of a link.

\subsection{Alternating genus}

Following his work on almost alternating links, Adams \cite{Adams:ToroidallyAlternating} defined {\em toroidally alternating links} as those links $L$ that have an alternating projection $\pi$ to a Heegaard torus $\Sigma$ such that the complement of the projection in the Heegaard torus, i.e. $\Sigma - \pi(L)$, is a disjoint union of disks. Toroidally alternating links can be naturally generalized. Define the {\em alternating genus} of a non-split link $L$, denoted $g_{\alt}(L)$, to be the minimum genus of any Heegaard surface $\Sigma$ such that the link $L$ has an alternating projection $\pi:S^3\to \Sigma$ and such that $\Sigma-\pi(L)$ is a disjoint union of disks. For any diagram $D$ of $L$, the Turaev surface $F(D)$ is an example of such a surface, and so $g_{\alt}(L)$ is well-defined, and 
\begin{equation}
\label{Inequality:TuraevAltGenus}
g_{\alt}(L)\leq g_T(L)
\end{equation}
for any link $L$.

\begin{proposition}
The alternating genus of a link is an alternating distance.
\end{proposition}
\begin{proof}
By definition, a non-split link $L$ has alternating genus zero if and only if it has an alternating projection to a sphere, i.e. if and only if $L$ is alternating. Moreover, since alternating genus is the minimum genus of some surface, it is always non-negative.  Let $L_1\#L_2$ be a connected sum of links $L_1$ and $L_2$ such that both $L_1$ and $L_2$ have an alternating projections $\pi_1$ and $\pi_2$ to respective Heegaard surfaces $\Sigma_1$ and $\Sigma_2$. Moreover, suppose that both $\Sigma_1-\pi_1(L_1)$ and $\Sigma_2-\pi_2(L_2)$ are disjoint unions of disks. Then there exists disks $D_1$ and $D_2$ in $\Sigma_1$ and $\Sigma_2$ meeting the projections $\pi(L_1)$ and $\pi(L_2)$ in a single arc such that $L_1\#L_2$ has an alternating projection to $\Sigma_1\#\Sigma_2$ whose complementary regions are a disjoint union of disks where the connected sum of $\Sigma_1$ and $\Sigma_2$ is taken along disks $D_1$ and $D_2$. Hence $g_{\alt}(L_1\#L_2)\leq g_{\alt}(L_1)+g_{\alt}(L_2)$, and therefore the alternating genus of a link is an alternating distance.
\end{proof}

\subsection{Warping span}

Shimizu \cite{Shimizu:WarpingKnot, Shimizu:WarpingLink} defines the warping degree of a knot or link diagram and uses warping degree to define the warping polynomial of a knot diagram \cite{Shimizu:WarpingPolynomial}. She defines the {\em span} of a knot $K$, denoted $\operatorname{spn}(K)$, to be the minimum span of the warping polynomial for any diagram of $K$.
We define a related invariant, called the warping span of $K$, that is essentially a renormalization of the span of the warping polynomial. 

Let $D$ be a knot diagram with $c>0$ crossings, and again let $\Gamma$ be the $4$-valent graph obtained from $D$ by forgetting the ``over-under" information at each crossing. An {\em edge} of $D$ is just an edge of $\Gamma$. Choose an orientation of $D$, and label the edges of $D$ by $e_1,e_2, \dots, e_{2c}$ where edge $e_1$ is chosen arbitrarily and edge $e_{i+1}$ follows edge $e_i$ with respect to the orientation of $D$. Assign a weight $d_i$ to each edge $e_i$ as follows. Set $d_1=0$, and set $d_{i+1}=d_i \pm 1$ according to the conventions of Figure \ref{figure:weights}.

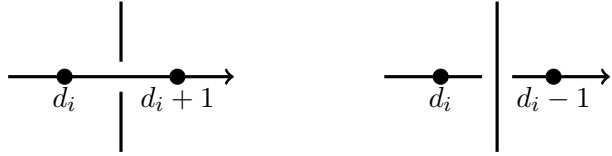
\begin{figure}[h]
$$\begin{tikzpicture}
\draw[very thick,->] (0,1) -- (3,1);
\draw[very thick] (1.5,2) -- (1.5,1.2);
\draw[very thick] (1.5,0.8) -- (1.5,0);
\draw[fill] (.75,1) circle (.1cm);
\draw[fill] (2.25,1) circle (.1cm);
\draw (.75,1) node[below]{$d_i$};
\draw (2.25,1) node[below]{$d_{i}+1$};

\begin{scope}[xshift = 5cm]
	\draw[very thick] (0,1) -- (1.3,1);
	\draw[very thick,->] (1.7,1) -- (3,1);
	\draw[very thick] (1.5,2) -- (1.5,0);
	\draw[fill] (.75,1) circle (.1cm);
	\draw[fill] (2.25,1) circle (.1cm);
	\draw (.75,1) node[below]{$d_i$};
	\draw (2.25,1) node[below]{$d_{i}-1$};
\end{scope}

\end{tikzpicture}$$
\caption{Weights on either side of a crossing.}
\label{figure:weights}
\end{figure}

Define the warping span of $D$ by $\warp(D)=\frac{1}{2}\max\{d_i-d_j-1~|~1\leq i,j \leq 2c\}$. The warping span of $D$ does not depend on the choice of orientation or the choice of initial edge $e_1$. In fact, any choice of weight for $d_1$ does not change $\warp(D)$. If $D$ does not have any crossings, i.e. $D$ is the crossingless diagram of the unknot, then define $\warp(D)=0$. The {\em warping span} of a knot $K$, denoted $\warp(K)$, is defined to be the minimum of $\warp(D)$ taken over all diagrams $D$ of $K$. For any nontrivial knot, the warping span and the span of $K$ are related via
$$\warp(K) = \frac{1}{2}\left(\operatorname{spn}(K)-1\right),$$
and for the unknot $U$, we have $\warp(U)=0$ while $\operatorname{spn}(U)=0$. An example of a diagram decorated with its weights is given in Figure \ref{figure:warpexample}.

\begin{figure}[h]
$$\begin{tikzpicture}[scale=.3]
\begin{scope}[thick, rounded corners = 2mm]
	\draw [->, ] (-3.5,.5) -- (-2.5,1.5);
	\draw (5.25, 1.25) -- (6.75,2.75);
	\draw (15.25,3.25) --(16,4) -- (21, -1) -- (16,-6) -- (0,-6) -- (-5,-1)--(0,4) -- (2,4) -- (6,0) -- (8,0)--(8.75, .75);
	\draw (1.25, 1.25) -- (2.75, 2.75);
	\draw (9.25, 1.25) -- (10.75, 2.75);
	\draw (3.25, 3.25) -- (4,4) -- (6,4) -- (10,0) -- (12,0) -- (12.75,.75);
	\draw (7.25,3.25) -- (8,4) -- (10,4) -- (14,0)-- (15,-1) -- (14,-2) -- (0,-2) -- (-1,-1) -- (0,0) -- (.75,.75);
	\draw (11.25, 3.25) -- (12,4) -- (14,4) -- (16,2) -- (19,-1) -- (16,-4) -- (0,-4)-- (-3,-1) -- (0,2) -- (2,0) -- (4,0) -- (4.75,.75);;
	\draw (13.25, 1.25) -- (14.75, 2.75);
	
	\draw (1,4) node[above]{$0$};
	\draw (5,4) node[above]{$0$};
	\draw (9,4) node[above]{$0$};
	\draw (13,4) node[above]{$0$};
	
	\draw (1.4,2.5) node{$1$};
	\draw (5.4,2.5) node{$1$};
	\draw (9.4,2.5) node{$1$};
	\draw (13.4,2.5) node{$1$};
	
	\draw (3.3, 1.7) node{$1$};
	\draw (7.3, 1.7) node{$1$};
	\draw (11.3, 1.7) node{$1$};
	\draw (15.3, 1.7) node{$1$};
	
	\draw (3,0) node[below]{$2$};
	\draw (7,0) node[below]{$2$};
	\draw (11,0) node[below]{$2$};
	\draw (7,-2) node[below]{$2$};

\end{scope}
\end{tikzpicture}$$
\caption{The $(3,4)$-torus knot together with its edge weights.}
\label{figure:warpexample}
\end{figure}
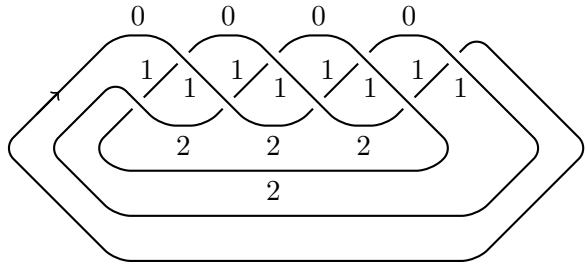

Warping span can be extended to apply to links. Suppose that $D$ is a diagram of a link with $\ell$ components $C_1,\dots, C_\ell$ such that each component has at least one crossing. Again let $\Gamma$ be the graph obtained from $D$ by forgetting the ``over-under" information at each crossing. Arbitrarily choose an orientation of $D$ and  one edge to label $e_1^k$ for each component $C_k$. Label the remaining edges so that $e_{i+1}^k$ follows $e_i^k$ with respect to the orientation of the $k$\textsuperscript{th} component $C_k$ of $D$. Suppose that component $C_k$ has $c_k$ edges for each $k$ with $1\leq k \leq \ell$. Assign the weight $d_1^k=0$ to each edge $e_1^k$, and set $d_{i+1}^k=d_i^k\pm 1$ according to the conventions of Figure \ref{figure:weights}. For each component $C_k$, define $w_k= \frac{1}{2}\max\{d_i^k-d_j^k-1~|~1\leq i,j \leq c_k\}$. Set $\warp(D)=\max\{w_k~|~1\leq k \leq \ell\}$. If $D$ is the standard diagram of an $\ell$-component unlink, then define $\warp(D)=0$. For any link diagram $D$, let $D\sqcup U$ be the disjoint union of $D$ and the standard diagram of the unknot, and define $\warp(D\cup U)=\warp(D)$. The warping span of the link $L$, denoted $\warp(L)$, is the minimum of $\warp(D)$ where $D$ is any diagram of $L$.

\begin{proposition}
The warping span of a link is an alternating distance.
\end{proposition}
\begin{proof}
Let $D$ be a diagram of $L$. If $D$ does not contain any crossings, then $\warp(D)=0$. If $D$ contains at least one crossing, then its edges contain at least two distinct weights, and hence $\warp(D)\geq 0$. Moreover, if $\warp(D)=0$, then every component of $D$ with a crossing contains exactly two distinct weights, which can happen if and only if $D$ is an alternating diagram. Hence $\warp(L)\geq 0$ and $\warp(L)=0$ if and only if $L$ is alternating. Shimizu \cite{Shimizu:WarpingPolynomial} proves that $\warp(K_1\#K_2)\leq \max\{\warp(K_1),\warp(K_2)\}$ for knots $K_1$ and $K_2$. Her argument also applies to the warping span of links, and so 
$$\warp(L_1\#L_2)\leq \max\{\warp(L_1),\warp(L_2)\}\leq \warp(L_1)+\warp(L_2)$$
for any connected sum $L_1\#L_2$ of links $L_1$ and $L_2$. Hence the warping span of a link is an alternating distance.
\end{proof}

Shimizu proves that changing a crossing in a knot diagram $D$ can alter $\warp(D)$ by at most one. The proof when $D$ is instead a link diagram is identical. Consequently,
\begin{equation}
\label{ineq:warpdalt}
\warp(L)\leq \dalt(L)
\end{equation}
for any link $L$.

\section{Obstructions and Lower Bounds}
\label{section:bounds}

The invariants under consideration are defined as a minimum over all diagrams or as a minimum over all projections to some surface. Typically, invariants of this form are difficult to compute, and so it will be useful to have obstructions and computable lower bounds for as many of the invariants as possible. The first obstruction comes from the hyperbolic geometry of the link complement. Menasco \cite{Menasco:AlternatingHyperbolic} proves that a prime, non-split alternating link is either a torus link or a hyperbolic link. Adams et al.\@ \cite{Adams:Almost} and Adams \cite{Adams:ToroidallyAlternating} extend this result to almost-alternating and toroidally alternating knots.
\begin{proposition}[Adams et al.\@]
\label{prop:hyperbolic}
 Let $K$ be a prime knot. If $\dalt(K)=1$ or $g_{\alt}(K)=1$, then $K$ is either a torus knot or a hyperbolic knot.
\end{proposition}
Since a knot has alternating genus and Turaev genus zero if and only if it is alternating, Inequality \ref{Inequality:TuraevAltGenus} implies the following corollary.
\begin{corollary}
Let $K$ be a prime knot. If $g_T(K)=1$, then $K$ is either a torus knot or a hyperbolic knot.
\end{corollary}

Our computable lower bounds arise from either Khovanov  homology \cite{Khovanov:homology} or knot Floer homology \cite{OzsvathSzabo:HolomorphicDisks, Rasmussen:HFK}. The Khovanov homology of a link $L$, denoted $Kh(L)$, is a bigraded $\mathbb{Z}$-module with homological grading $i$ and polynomial (or quantum) grading $j$. The diagonal grading $\delta$ is defined by $\delta=j-2i$, and when $Kh(L)$ is decomposed over summands with respect to the $\delta$-grading, we write $Kh(L) = \bigoplus_{\delta} Kh^\delta(L).$ The {\em width of $Kh(L)$}, denoted $w(Kh(L))$, is defined as
$$w(Kh(L)) =\frac{1}{2} \left(\max\{\delta|Kh^\delta(L)\neq 0\} - \min\{\delta|Kh^\delta(L)\neq 0\}\right) +1.$$
The factor of $1/2$ is included in the definition of width since if $L$ has an odd number of components, all $\delta$-gradings where $Kh^\delta(L)\neq 0$ are even, and if $L$ has an even number of components, all $\delta$-gradings where $Kh^\delta(L)\neq 0$ are odd. Champanerkar and Kofman \cite{ChampanerkarKofman:SpanningTrees} and independently Wehrli \cite{Wehrli:SpanningTrees} show that there is a complex whose generators correspond to spanning trees of the checkerboard graph of a diagram $D$ of $L$ and whose homology is the Khovanov homology $Kh(L)$ of $L$. Champanerkar, Kofman, and Stotlzfus \cite{CKS:DessinsKhovanov} use a relationship between spanning trees of the checkerboard graph of $D$ and certain graphs embedded in the Turaev surface $F(D)$ to show
\begin{equation}
\label{ineq:TuraevKhov}
w(Kh(L)) - 2 \leq g_T(L).
\end{equation}
The relationship between the Turaev surface and Khovanov homology is further explained by Dasbach and the author \cite{DL:Ribbon}.

Let $\mathbb{F}$ denote the vector space with two elements. We consider the ``hat version" $\widehat{HFK}(K)$ of the knot Floer homology of a knot $K$ with coefficients in $\mathbb{F}$. The invariant $\widehat{HFK}(K)$ is a bigraded $\mathbb{F}$-vector space with Maslov (or homological) grading $m$ and Alexander (or polynomial) grading $s$. The diagonal grading $\delta$ is defined as $\delta=s-m$, and when $\widehat{HFK}(K)$ is decomposed over summands with respect to the $\delta$-grading, we write $\widehat{HFK}(K) = \bigoplus_{\delta}\widehat{HFK}_\delta(K).$ The {\em width of $\widehat{HFK}(K)$}, denoted $w(\widehat{HFK}(K))$, is defined as
$$w(\widehat{HFK}(K))= \max\{\delta|\widehat{HFK}_\delta(K)\neq 0\} - \min\{\delta | \widehat{HFK}_\delta(K)\neq 0\} + 1.$$
Ozsv\'ath and Szab\'o \cite{OzsvathSzabo:Alternating} show that there is a complex whose generators correspond to spanning trees of the checkerboard graph of a diagram $D$ of $K$ and whose homology is $\widehat{HFK}(K)$. The author \cite{Lowrance:WidthTuraevGenus} uses the \OS~spanning tree complex to prove that 
\begin{equation}
\label{inequality:HFKTuraev}
w(\widehat{HFK}(K)) - 1\leq g_T(K).
\end{equation}

Rasmussen \cite{Rasmussen:KhovanovSlice} uses Lee's spectral sequence \cite{Lee:KhovanovAlternating} to show that the Khovanov complex of a knot $K$ gives rise to a concordance invariant $s(K)$. Suppose that $K_+$ and $K_-$ are two knots such that $K_+$ can be transformed into $K_-$ by changing a single positive crossing to a negative crossing in some diagram (see Figure \ref{figure:sign}). 
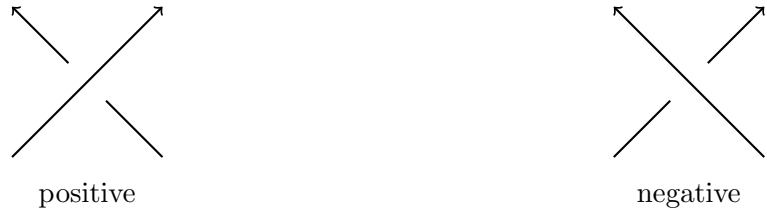
\begin{figure}[h]
$$\begin{tikzpicture}

\draw [thick, ->] (0,0) -- (2,2);
\draw [thick, ->] (.75,1.25) -- (0,2);
\draw [thick] (1.25,.75) -- (2,0);
\draw (1,-.5) node{positive};

\begin{scope}[xshift = 8cm]
\draw [thick, ->] (2,0) -- (0,2);
\draw [thick, ->] (1.25,1.25) -- (2,2);
\draw [thick] (0,0) -- (.75, .75);
\draw (1,-.5) node{negative};
\end{scope}

\end{tikzpicture}$$
\caption{{\bf Left:} a positive crossing. {\bf Right:} a negative crossing.}
\label{figure:sign}
\end{figure}

Then
\begin{equation}
\label{inequality:Rasmussen}
s(K_-)\leq s(K_+) \leq s(K_-)+2.
\end{equation}
Let $\sigma(K)$ denote the signature of a knot with sign convention chosen so that the signature of the positive trefoil is $-2$. Cochran and Lickorish \cite{CochranLickorish:Unknotting} show that
\begin{equation}
\label{inequality:sigcross}
\sigma(K_-)-2\leq \sigma(K_+)\leq \sigma(K_-).
\end{equation}
Abe \cite{Abe:AlternationNumber} uses the behavior of the $s$-invariant and signature under crossing changes to show
\begin{equation}
\label{inequality:abealt}
|s(K) + \sigma(K)| \leq 2\alt(K).
\end{equation}

Using work on knot signature of Murasugi \cite{Murasugi:InvariantsOfGraphs} and Thistlethwaite \cite{Thistlethwaite:AdequateKauffman} and the spanning tree complexes of Champankerkar-Kofman and Wehrli, Dasbach and the author \cite{DL:Concordance} show that
\begin{equation}
\label{inequality:DLRasSig}
|s(K) + \sigma(K)| \leq 2g_T(L).
\end{equation}

\section{Some families of links and their alternating distances}
\label{section:examples}

In this section we examine the three families of links: iterated Whitehead doubles of the figure-eight knot, modified torus links, and the $(3,q)$-torus links. We use these three families to prove Theorems \ref{theorem:main1} and \ref{theorem:main2}.

\subsection{Iterated Whitehead doubles of the figure-eight knot}

Let $P$ be a knot embedded in a genus one handlebody $Y$. For any knot $K$, identify a regular neighborhood of $K$ with $Y$ such that the generator of $H_1(Y,\mathbb{Z})$ is identified with a longitude of $K$ coming from a Seifert surface. The image of $P$ is a knot $S$, called a satellite of $K$. The knot $P$ is the {\em pattern} for $S$, and $K$ is the {\em companion}. Define the {\em positive $t$-twisted Whitehead double} of a knot $K$, denoted $D_+(K,t)$, to be the satellite of $K$ where the pattern knot $P$ is the $t$-twisted positive clasp knot given in Figure \ref{figure:Whitehead}. We use certain iterated Whitehead doubles of the figure-eight knot to show that the Turaev genus and dealternating number dominate the alternation number of a link.
\begin{figure}[h]
$$\begin{tikzpicture}[scale=.7]
\begin{scope}[yshift=8cm, xshift=.8cm]

\draw [very thin] (0,0) ellipse (3cm and 2cm);
\begin{scope}[yshift=.3cm]
	\draw [very thin] (0:1cm) arc (0:-180: 1cm and .5cm);
\end{scope}
\draw [very thin] (0:.8cm) arc (0:180 :.8cm and .4cm);

\begin{pgfonlayer}{background2}
\draw [very thick] (2.5,0) arc (0:-270: 2.5cm and 1.725cm);
\draw [very thick] (2.5,0) arc (0: 90: 2.5cm and 1.5cm);
\draw [very thick] (0,1.375) arc (90:180: 2cm and 1.375cm);
\draw [very thick] (-2,0) arc (-180:0:2cm and 1.375cm);
\draw [very thick] (2,0) arc (0:90:2cm and 1.25cm);
\draw [very thick] (0,1.725) arc (90:10:.2cm);
\draw [very thick] (0,1.375) arc (-90: -45:.2cm);
\draw [very thick] (0,1.5) arc (90:140:.2cm);
\draw [very thick] (0,1.25) arc (-90:-140:.2cm);
\end{pgfonlayer}

\begin{pgfonlayer}{background}
\fill[white] (-.5,-1.8) rectangle (.5, -1.2);
\draw (-.5,-1.8) rectangle (.5, -1.2);
\end{pgfonlayer}

\draw (0,-1.5) node{$t$};

\end{scope}

\begin{scope}[very thick, scale=.4]
	\draw [rounded corners = 3mm] (1.5,1.5) -- (0,0) -- (-2,0) -- (-4,4.5) -- (-2,9) -- (3.5,9);
	\draw [rounded corners = 3mm] (1.5,5.5) -- (0,4) -- (4,0) -- (6,0) -- (8,4.5) -- (6,9) -- (4.5,9);
	\draw [rounded corners = 2mm] (2.5,2.5) -- (4,4) -- (0,8) --(0,8.5);
	\draw [rounded corners = 2mm] (2.5,6.5) -- (4,8) -- (4,10) -- (2,12) -- (0,10) -- (0,9.5);
\end{scope}

\begin{scope}[scale = .6, xshift= 14cm, yshift=4cm, rounded corners = 3mm]
\begin{pgfonlayer}{background2}
	\draw [very thick]  (1.6,1.4) -- (0,-.2) -- (-2,-.2) -- (-4,4.5) -- (-2,9.2) -- (2,9.2);
	\draw [very thick] (1.4, 8.8) -- (-1.8,8.8) -- (-3.6,4.5) -- (-1.8,.2) -- (0,.2) -- (1.5, 1.7);
	\draw [very thick] (1.8, 8.8) -- (2,8.8);
	\draw [very thick] (3.5, 9) -- (2.4,9);
	\draw [very thick] (2,9) -- (1.8,9);
	\draw [very thick] (1.8,9) arc (90:270:.2cm);
	\draw [very thick] (3.5, 8.6) -- (1.8, 8.6);
	\draw [very thick] (1.4,5.6) -- (-.2,4) -- (3.8,-.2) -- (6,-.2) -- (8,4.5) -- (6,9) -- (4.5,9);
	\draw [very thick](4.5, 8.6) -- (5.6,8.6) -- (7.6,4.5) -- (5.6,.2) -- (4,.2) -- (.2,4) -- (1.7,5.5);
	\draw  [very thick](2.6,2.4) -- (4.2,4) -- (0.2,8) --(0.2,8.5);
	\draw [very thick] (2.3,2.5) -- (3.8,4) -- (-.2,8) -- (-.2,8.5);
	\draw  [very thick](2.5,6.3) -- (4.2,8) -- (4.2,10) -- (2,12.2) -- (-.2,10) -- (-.2,9.5);
	\draw [very thick](2.3, 6.5) -- (3.8, 8) -- (3.8,10) -- (2,11.8) -- (.2,10) -- (.2,9.5);
	\draw [very thick](2,9.2) arc (90:-90:.2cm);
\end{pgfonlayer}
\begin{pgfonlayer}{background}
\begin{scope}[rounded corners = 0mm]
	\fill[white] (-1.5,.5) rectangle (0,-.4);
	\draw (-1.5,.5) rectangle (0,-.4);
\end{scope}
\end{pgfonlayer}

\draw (-.75,.05) node{$t$};

\end{scope}

\draw (-.7,8) node{$P$};
\draw (4.5, 8) node{$Y$};
\draw (4,2) node{$K$};
\draw (9.5, 1.5) node{$D_+(K,t)$};

\end{tikzpicture}$$
\caption{The satellite of the figure-eight $K$ with pattern $P$ and companion $K$ is the positive $t$-twisted Whitehead double $D_+(K,t)$ of $K$. Each box labeled $t$ indicates $t$ positive full twists.}
\label{figure:Whitehead}
\end{figure}
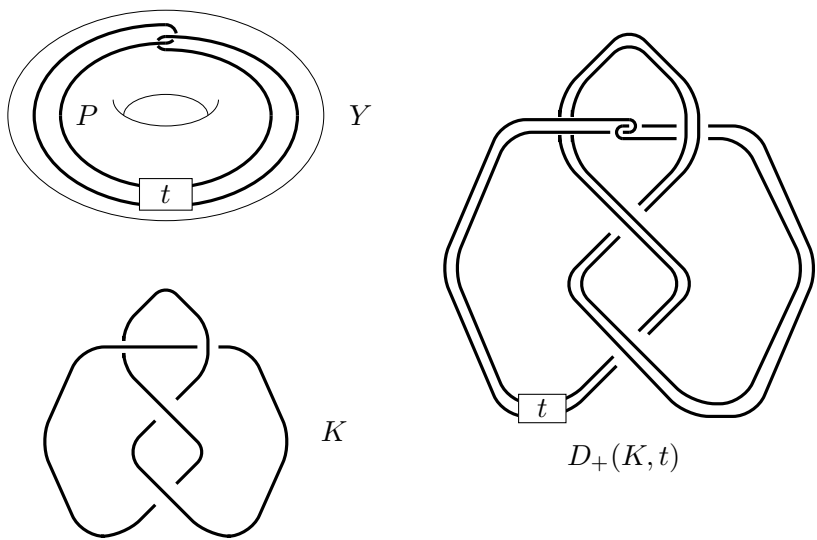

Let $W_0$ be the figure-eight knot $4_1$. For each positive integer $n$, define $W_n=D_+(W_{n-1},0)$, that is $W_n$ is the positive $n$-th iterated untwisted Whitehead double of the figure-eight knot. Define the family $\mathcal{F}(W_n)$ by $\mathcal{F}(W_n) = \{W_n~|~n\geq 0, n\in\mathbb{Z}\}$. Hedden \cite{Hedden:Whitehead} computes the knot Floer homology of $W_n$.

\begin{proposition}[Hedden]
\label{prop:Hedden}
Let $\mathbb{F}$ denote the field with two elements, and let $\mathbb{F}^k_{(m)}$ denote the vector space $\mathbb{F}^k$ in homological grading $m$. Then
$$\widehat{HFK}_*(W_n,s)\cong
\begin{cases}
\bigoplus_{m=0}^n\mathbb{F}^{2^{n}\binom{n}{m}}_{(1-m)} & s=1,\\
\mathbb{F}_{(0)}\bigoplus_{m=0}^n \mathbb{F}^{2^{n+1}\binom{n}{m}}_{(-m)} & s=0,\\
\bigoplus_{m=0}^n\mathbb{F}^{2^{n}\binom{n}{m}}_{(-1-m)} & s=-1,\\
0 & \text{otherwise,}
\end{cases}$$
and $w(\widehat{HFK}(W_n))=n+1$.
\end{proposition}

The alternation number, Turaev genus, and alternating genus of $W_n$ behave according to the following proposition.
\begin{proposition} 
\label{prop:Whitehead}
For each positive integer $n$,
\begin{align*}
\operatorname{alt}(W_n) & = 1,\\
g_T(W_n) & \geq n,~\text{and} \\
 g_{\operatorname{alt}}(W_n) & > 1.
\end{align*}
\end{proposition}
\begin{proof} Let $n$ be a positive integer. Changing one of the crossings of the clasp in any Whitehead double transforms the knot into an unknot, and thus the unknotting number $u(W_n)$ is one. Since the unknot is alternating, the inequality $\operatorname{alt}(K)\leq u(K)$ holds for every knot, and hence $\operatorname{alt}(W_n)\leq 1$. The knot $W_n$ is non-alternating since $w(\widehat{HFK}(W_n))>1$, and so $\operatorname{alt}(W_n)=1$.

Inequality \ref{inequality:HFKTuraev} states that for any knot $K$, we have $w(\widehat{HFK}(K))-1\leq g_T(K)$, and hence Proposition \ref{prop:Hedden} implies that $g_T(W_n)\geq n$.

Since the genus of $W_n$ is one, it follows that $W_n$ is prime, and because $W_n$ is a satellite knot, Proposition \ref{prop:hyperbolic} implies that $g_{\operatorname{alt}}(W_n)\neq 1$. Since $W_n$ is non-alternating, we may conclude that $g_{\operatorname{alt}}(W_n)>1$.
\end{proof}

\subsection{Modified torus links}
The modified torus link $\widetilde{T}(p,q)$ is obtained by changing certain crossings of a standard diagram of the $(p,q)$-torus link $T(p,q)$. We use the natural embedding of $T(p,q)$ on a torus to show that $\widetilde{T}(p,q)$ is toroidally alternating for many choices of $p$ and $q$. The behavior of the Rasmussen $s$-invariant and knot signature under crossing changes imply that the Turaev genus and alternation number of $\widetilde{T}(p,q)$ can be arbitrarily large.

Let $B_p$ denote the $p$-stranded braid group, let $\Delta_p\in B_p$ denote the braid $\sigma_1\sigma_2\cdots\sigma_{p-1}$, and let $\widetilde{\Delta}_p$ denote the braid 
$$\widetilde{\Delta}_p = \prod_{i=1}^{p-1} \sigma_{i}^{(-1)^{i+1}} = \sigma_1\sigma_2^{-1}\sigma_3\cdots\sigma_{p-1}^{(-1)^p}.$$
Define $\widetilde{T}(p,q)$ to be the closure of the braid $\Delta_p^{q-1}\widetilde{\Delta}_p$. The link $\widetilde{T}(p,q)$ can be obtained from the diagram of the closure of $\Delta_p^q$ by changing $\lfloor \frac{p-1}{2}\rfloor$ crossings where the closure of $\Delta_p^q$ is a familiar diagram of the $(p,q)$-torus link $T(p,q)$. Define the family $\mathcal{F}(\widetilde{T}(p,q))$ by $$\mathcal{F}(\widetilde{T}(p,q))=\{\widetilde{T}(p,q)~|~ p,q \geq 3, p,q\in\mathbb{Z}\}.$$

Gordon, Litherland, and Murasugi \cite{GLM:TorusSignature} give the following recursive algorithm for computing the signature of the torus link $T(p,q)$.
\begin{theorem}[Gordon, Litherland, Murasugi]
\label{theorem:sigtorus}
Suppose that $p,q>0$. The following recurrence formulas hold.
\begin{enumerate}
\item Suppose that $2p<q$.
	\begin{enumerate}
	\item If $p$ is odd, then $\sigma(T(p,q))=\sigma(T(p,q-2p))-p^2+1$.
	\item If $p$ is even, then $\sigma(T(p,q))=\sigma(T(p,q-2p))-p^2$.
	\end{enumerate}
\item $\sigma(T(p,2p))=1-p^2$.
\item Suppose that $p\leq q < 2p$. 
	\begin{enumerate}
	\item If $p$ is odd, then $\sigma(T(p,q)) = 1 - p^2 - \sigma(T(p,2p-q))$.
	\item If $p$ is even, then $\sigma(T(p,q)) = 2 - p^2 - \sigma(T(p,2p-q))$.
	\end{enumerate}
\item $\sigma(T(p,q))=\sigma(T(q,p)),\quad \sigma(T(p,1))=0,\quad \sigma(T(2,q))=1-q.$
\end{enumerate}
\end{theorem}

In order to estimate the bounds in Inequalities \ref{inequality:abealt} and \ref{inequality:DLRasSig}, we first estimate the signature and Rasmussen invariant for torus knots and the modified torus knots $\widetilde{T}(p,q)$.
\begin{proposition}
\label{proposition:ModifiedTorus}
Let $p$ and $q$ be relatively prime integers with $p\geq 3$ and $q\geq 3$. Then
\begin{align}
\label{ineq1}
& -(p-1)(p-2) - \frac{1}{2}pq  \leq  \sigma(T(p,q))  \leq  (p-1)(p-2) - \frac{1}{2}(p-1)q, \\
\label{ineq2}
&-(p-1)(p-2) - \frac{1}{2}pq  \leq  \sigma(\widetilde{T}(p,q))  \leq  (p-1)^2 - \frac{1}{2}(p-1)q,~\text{and}\\
\label{ineq3}
& pq-2p-q+2  \leq  s(\widetilde{T}(p,q))  \leq  pq-p-q+1.
\end{align}
\end{proposition}
\begin{proof}
Let $k$ be an integer relatively prime to $p$ with $0 < k < p$. Rudolph \cite{Rudolph:PositiveSignature} shows that the signature of the closure of a positive braid is negative, and hence $\sigma(T(p,k))\leq 0$. The unknotting number $u(K)$ of any knot $K$ satisfies $|\sigma(K)|\leq 2u(K)$.
Kronheimer and Mrowka \cite{KronMrowka:Gauge} show that the unknotting number of $T(p,k)$ is $\frac{1}{2}(p-1)(k-1)$, and thus 
\begin{equation}
\label{inequality:torussig}
-(p-1)(p-2)\leq \sigma(T(p,k))\leq 0.
\end{equation}

Let $q=np+r$ where $n$ is a non-negative integer and $0<r<p$. Suppose that $n=0$. Then $(p-1)(p-2) - \frac{1}{2}(p-1)q = (p-1)(p-2)-\frac{1}{2}(p-1)r \geq 0$ because $\frac{1}{2}r\leq p-2$. Thus
\begin{align*}
-(p-1)(p-2) - \frac{1}{2}pq & = -(p-1)(p-2) -\frac{1}{2}pr\\
& < -(p-1)(p-2)\\
&\leq \sigma(T(p,q))\\
&\leq 0\\
&\leq (p-1)(p-2) - \frac{1}{2}(p-1)q,
\end{align*}
and Inequality \ref{ineq1} is proven when $n=0$.

Now suppose that $n>0$. Repeated applications of Theorem \ref{theorem:sigtorus} yields
\begin{equation}
\label{equation:sigtorus}
\sigma(T(p,np+r)) = \begin{cases}
\sigma(T(p,r))-\frac{1}{2}n(p^2-1) & p~\text{odd,}~n~\text{even},\\
\sigma(T(p,r))- \frac{1}{2}np^2 & p~\text{even,}~n~\text{even},\\
-\sigma(T(p,p-r))-\frac{1}{2}(n+1)(p^2-1) & ~p~\text{odd,}~n~\text{odd},\\
-\sigma(T(p,p-r))-\frac{1}{2}(n+1)p^2 + 2& ~p~\text{even,}~n~\text{odd}.
\end{cases}
\end{equation}
Since $0< r < p$ and $0< p-r < p$, Inequality \ref{inequality:torussig} implies that $-(p-1)(p-2)\leq \sigma(T(p,r))\leq 0$ and $-(p-1)(p-2)\leq \sigma(T(p,p-r))\leq 0$. Combining these inequalities with Equation \ref{equation:sigtorus} yields the inequalities
$$-(p-1)(p-2) - \frac{1}{2}np^2 - 1\leq \sigma(T(p,q)) \leq (p-1)(q-1) - \frac{1}{2}(n+1)(p^2-1).$$
Since $q=np+r$, it follows that $pq=np^2 + rp > np^2+2$, and hence 
$$-(p-1)(p-2) - \frac{1}{2}pq < -(p-1)(p-2) - \frac{1}{2}np^2 -1 \leq \sigma(T(p,q)).$$
Likewise, since $q=np+r$, it follows that 
\begin{align*}
(p-1)q & = np^2 + rp -np - r\\
& \leq np^2 + p^2 - n -1\\
&=(n+1)(p^2-1).
\end{align*}
Hence
$$\sigma(T(p,q))\leq (p-1)(q-1) - \frac{1}{2}(n+1)(p^2-1) \leq (p-1)(p-2) - \frac{1}{2}(p-1)q,$$
and thus Inequality \ref{ineq1} is proven for all $n$.

Since $\widetilde{T}(p,q)$ can be obtained from $T(p,q)$ via $\lfloor\frac{p-1}{2}\rfloor$ crossing changes, Inequality \ref{ineq2} follows from Inequalities \ref{inequality:sigcross} and \ref{ineq1}. Rasmussen computes the $s$-invariant for positive knots, and applying his formula to $T(p,q)$ yields $s(T(p,q)) = pq-p-q+1$. Inequality \ref{ineq3} follows from this fact and Inequality \ref{inequality:Rasmussen}.
\end{proof}

\begin{corollary}
\label{cor:modTor}
Let $p$ be a fixed integer with $p\geq 3$. For any positive integer $n$, there exists a $q$ relatively prime to $p$ such that $\alt(\widetilde{T}(p,q))\geq n$ and $g_T(\widetilde{T}(p,q))\geq n$.
\end{corollary}
\begin{proof}
Proposition \ref{proposition:ModifiedTorus} implies that if $p$ and $q$ are relatively prime with $p\geq 3$ and $q\geq 3$ then
$$\frac{1}{2}pq-2p-q+2-(p-1)(p-2) \leq s(\widetilde{T}(p,q))+\sigma(\widetilde{T}(p,q))$$
If $p$ is fixed and $q$ goes to infinity, then the left hand side of the previous inequality is eventually positive and grows without bound. Therefore for sufficiently large values of $q$ (with $p$ and $q$ relatively prime), we have
$$\frac{1}{2}pq-2p-q+2-(p-1)(p-2) \leq |s(\widetilde{T}(p,q))+\sigma(\widetilde{T}(p,q))|.$$
Hence inequalities \ref{inequality:abealt} and \ref{inequality:DLRasSig} imply that $\operatorname{alt}(\widetilde{T}(p,q))$ and $g_T(\widetilde{T}(p,q))$ can be arbitrarily large.
\end{proof}

An alternate proof that $g_T(\widetilde{T}(p,q))$ can be arbitrarily large uses Inequality \ref{ineq:TuraevKhov} and a computation of the Khovanov width of $\widetilde{T}(p,q)$. While the alternation number and Turaev genus of $\widetilde{T}(p,q)$ grow without bound, the next proposition shows that many such knots are either alternating or toroidally alternating.

\begin{proposition}
\label{prop:AltGenModTor}
Let $p$ be an even integer and let $q$ be an odd integer with $p\geq 4$, $q\geq 3$, and $q\neq 1 \mod p$. Then
$$g_{\alt}(\widetilde{T}(p,q))\leq 1.$$
\end{proposition}
\begin{proof}
Let $\Sigma$ be a Heegaard torus in $S^3$ and suppose that $T(p,q-1)$ is embedded on $\Sigma$ by going $p$ times around the longitude and $q-1$ times around the meridian. Cut $\Sigma$ along a meridian to obtain a cylinder $C_1$. Suppose a standard planar diagram of $\widetilde{\Delta}_p$ is embedded on a cylinder $C_2$ so that the incoming strands meet one boundary component and the outgoing strands meet the other. Identify each boundary component of $C_1$ with a boundary component of $C_2$ so that the resulting surface $\Sigma$ is again a Heegaard torus in $S^3$ with a diagram of $\widetilde{T}(p,q)$ projected onto it (via some projection $\pi$). 

Before we glue in the cylinder $C_2$ each component of $\Sigma - T(p,q-1)$ is an annulus, and since $q\neq p \mod 1$, the number of components of $\Sigma-T(p,q-1)$ is strictly less than $p$. By surgering in the diagram of $\widetilde{\Delta}_p$, the annuli are separated into a disjoint union of disks.

Label the incoming and outgoing strands of $\widetilde{\Delta}_p$ by $1,2,\dots, p$ from left to right. The incoming strands labeled with odd numbers encounter an under-crossing first, while the incoming strands labeled with even integers encounter an over-crossing first. For the outgoing strands, the situation is reversed. Outgoing strands labeled with odd numbers most recently encountered over-crossings, while outgoing strands labeled with even integers most recently encountered under-crossings. Since $p$ is even and $q-1$ is even, the permutation on the strands induced by $\Delta_p^{q-1}$ sends strands labeled with odd numbers to strands labeled with odd numbers and sends strands labeled with even numbers to strands labeled with even numbers. Therefore, the projection $\pi(\widetilde{T}(p,q))$ is alternating on $\Sigma$.

Since $\widetilde{T}(p,q)$ has an alternating projection to a Heegaard torus where the complement of the projection is a disjoint union of disks, it follows that $g_{\alt}(\widetilde{T}(p,q))\leq 1$.
\end{proof}

Figure \ref{figure:TorAlt} shows that $\widetilde{T}(4,3)$ and $\widetilde{T}(4,4k+3)$ for non-negative $k$ have alternating projections to Heegaard tori in $S^3$ whose complements are disjoint unions of disks. 
\begin{figure}[h]
$$\begin{tikzpicture}[scale = .8]
\begin{pgfonlayer}{background}
\fill [white](2.5,20) ellipse (2.5cm and .5cm);
\end{pgfonlayer}

\draw (2.5,20) ellipse (2.5cm and .5cm);
\draw (0,9) arc (180:360:2.5cm and .5cm);
\draw (0,9) -- (0,20);
\draw (5,9) -- (5,20);

\begin{pgfonlayer}{background2}
\begin{scope}[thick, rounded corners = 4mm]
	\draw (1,20) -- (1,19) -- (0,18);
	\draw [dashed] (0,18) -- (5,17);
	\draw (5,17) -- (4,16) -- (3,13) -- (3,12) -- (2.7,11.7);
	\draw (2.3, 11.3) -- (2,11) -- (2,8.5);
	\draw (2,20) -- (2,19) -- (1,16) -- (0,15);
	\draw [dashed] (0,15) -- (5,14);
	\draw (3,20) -- (3,19) -- (2,16) -- (1,13) -- (1.3,12.7);
	\draw (4,20) -- (4,19) -- (3,16) -- (2,13) -- (1,12) -- (1,8.6);
	\draw (5,14) -- (4,13) -- (4,11) -- (3,10) -- (3,8.5);
	\draw (1.7,12.3) -- (3.3, 10.7);
	\draw (3.7, 10.3) -- (4,10) -- (4,8.6);
\end{scope}
\end{pgfonlayer}

\draw (2.5, 21.5) node{$\widetilde{T}(4,3)$};

\begin{scope}[xshift = 10cm]
\begin{pgfonlayer}{background}
\fill [white](2.5,20) ellipse (2.5cm and .5cm);
\end{pgfonlayer}

\draw (2.5,20) ellipse (2.5cm and .5cm);

\draw (0,19) arc (180:360:2.5cm and .5cm);
\draw (0,14) arc (180:360:2.5cm and .5cm);

\draw (0,9) arc (180:360:2.5cm and .5cm);
\draw (0,9) -- (0,20);
\draw (5,9) -- (5,20);

\begin{pgfonlayer}{background2}
\begin{scope}[thick, rounded corners = 4mm]
	\draw (1,20) -- (1, 18.6);
	\draw (2,20) -- (2, 18.5);
	\draw (3,20) -- (3, 18.5);
	\draw (4,20) -- (4,18.6);
	\draw (1.3,12.7) -- (1,13) -- (1,13.6);
	\draw (2,13.4) --  (2,13) -- (1,12) -- (1,8.6);
	\draw (3,13.4) -- (3,13) -- (3,12) -- (2.7,11.7);
	\draw (4,13.6) -- (4,11) -- (3,10) -- (3,8.5);
	\draw (2.3, 11.3) -- (2,11) -- (2,8.5);
	\draw (1.7,12.3) -- (3.3, 10.7);
	\draw (3.7, 10.3) -- (4,10) -- (4,8.6);
	
	\draw (2.5,16.8) node{$4k+2$};
	\draw (2.5,16.2) node{full twists};
\end{scope}
\end{pgfonlayer}

\draw (2.5, 21.5) node{$\widetilde{T}(4,4k+3)$};

\end{scope}

\end{tikzpicture}$$
\caption{After identifying the two components of the boundary of each cylinder to obtain tori, the diagram on the left is an alternating projection of $\widetilde{T}(4,3)$ to the torus, and the diagram on the right is an alternating projection of $\widetilde{T}(4,4k+3)$ to the torus for each $k\geq 0$.}
\label{figure:TorAlt}
\end{figure}
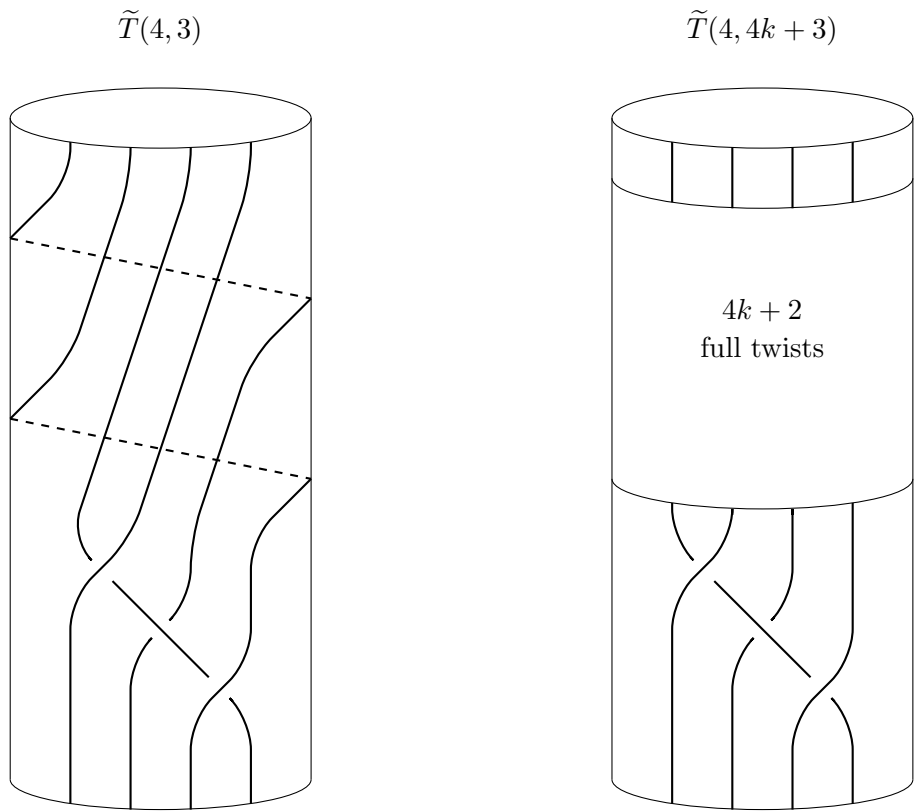

\subsection{The $(3,q)$-torus knots} Let $\mathcal{F}(T(3,q))$ be the family of $(3,q)$-torus links where $q$ is any integer. The $(3,q)$-torus knots have arbitrarily large alternation number, Turaev genus, $c(T(3,q)) - \Span_{T(3,q)}(t)$, and dealternating number, but have warping span $0$ or $1/2$. Kanenobu \cite{Kanenobu:Alternation} computes the alternation numbers of the $(3,q)$-torus knots, up to an additive error of at most one.
\begin{proposition}[Kanenobu]
\label{prop:AltTor3}
For any positive integer $n$,
\begin{align*}
&\alt(T(3,4))=\alt(T(3,5))=1,\\
& \alt(T(3,6n+1))=\alt(T(3,6n+2))=2n,\\
& \alt(T(3,6n+4))=\alt(T(3,6n+5))=2n~\text{or}~2n+1.
\end{align*}
\end{proposition}

Using Inequality \ref{ineq:TuraevKhov} and work of Sto{\v{s}}i{\'c} \cite{Stosic:Torus} and Turner \cite{Turner:Torus}, the author \cite{Lowrance:Twisted} computes the Turaev genus of the $(3,q)$-torus knots. Abe and Kishimoto \cite{Abe:Dealternating} independently compute the Turaev genus of the $(3,q)$-torus knots and also compute their dealternating numbers.
\begin{proposition}[Abe-Kishimoto, Lowrance]
\label{prop:TuraevDaltTor3}
Let $n$ be a non-negative integer, and let $i=1$ or $2$. Then
$$g_T(T(3,3n+i)) = \dalt(T(3,3n+i))=n.$$
\end{proposition}

Combining work of Jones and Murasugi, we obtain the following result about the difference between the crossing number and the span of the Jones polynomial of $T(3,q)$.
\begin{proposition}[Jones, Murasugi]
\label{prop:SpanTor3}
Suppose that $q$ is relatively prime to $3$ and $q>3$. Then
$$c(T(3,q)) - \Span V_{T(3,q)}(t) = q-1.$$
\end{proposition}
\begin{proof}
Jones \cite{Jones:Hecke} gives the following formula for the Jones polynomial of all torus knots:
$$V_{T(p,q)}(t) = \frac{t^{(p-1)(q-1)/2}}{1-t^2}(1-t^{p+1} - t^{q+1} + t^{p+q}).$$
 For the $(3,q)$-torus knots, this formula yields
$$V_{T(3,q)}(t) = t^{q-1}+t^{q+1}-t^{2q}.$$
Hence $\Span V_{T(3,q)}(t) = q+1$.
Murasugi \cite{Murasugi2} shows that the crossing number of $T(3,q)$ where $q>3$ is $2q$. Therefore
$$c(T(3,q)) - \Span V_{T(3,q)}(t) = q-1.$$
\end{proof}

An example of Shimizu \cite{Shimizu:WarpingPolynomial} is easily generalized to the following result.
\begin{proposition}
\label{prop:WarpTor3}
If $q$ is an integer with $|q|>2$, then $\warp(T(3,q))=\frac{1}{2}$.
\end{proposition}
\begin{proof}
Suppose $q>0$. We write $T(3,q)$ as the closure of $(\sigma_1\sigma_2)^q$, and $T(3,-q)$ as the closure of $(\sigma_1^{-1}\sigma_2^{-1})^q$. Figure \ref{figure:torus} depicts the $3$-braids $\sigma_1\sigma_2$ and $\sigma_1^{-1}\sigma_2^{-1}$ with edges labeled by weights $0$, $1$, and $2$. Since the incoming weights are the same as the outgoing weights, these braids can be stacked $q$ times to obtain diagrams of $T(3,q)$ and $T(3,-q)$ where the only weights are $0$, $1$, and $2$. Hence $\warp(T(3,q))=\warp(T(3,-q))\leq\frac{1}{2}$. Since $T(3,q)$ is alternating if and only if $|q|\leq 2$, the result follows.
\end{proof}

\begin{figure}[h]
$$\begin{tikzpicture}[rounded corners = 3mm, scale=.7]
\draw (0,4) -- (.75,3.25);
\draw (1.25, 2.75) -- (2.75, 1.25);
\draw (3.25, .75) -- (4,0);
\draw (2,4) -- (0,2) -- (0,0);
\draw (4,4) -- (4,2) -- (2,0);
\draw (0,4) node[above]{$2$};
\draw (2,4) node[above]{$1$};
\draw (4,4) node[above]{$0$};
\draw (0,0) node[below]{$2$};
\draw (2,0) node[below]{$1$};
\draw (4,0) node[below]{$0$};
\draw (2,2.2) node[above]{$1$};

\begin{scope}[xshift = 8cm]
\draw (0,4) -- (4,0);
\draw (2,4) -- (1.25, 3.25);
\draw (.75, 2.75) -- (0,2) -- (0,0);
\draw (4,4) -- (4,2) -- (3.25,1.25);
\draw (2.75,.75) -- (2,0);
\draw (0,4) node[above]{$0$};
\draw (2,4) node[above]{$1$};
\draw (4,4) node[above]{$2$};
\draw (0,0) node[below]{$0$};
\draw (2,0) node[below]{$1$};
\draw (4,0) node[below]{$2$};
\draw (2,2.2) node[above]{$1$};
\end{scope}
\end{tikzpicture}$$
\caption{The $3$-braids $\sigma_1\sigma_2$ (left) and $\sigma_1^{-1}\sigma_2^{-1}$ (right) with their edges labeled by weights.}
\label{figure:torus}
\end{figure}
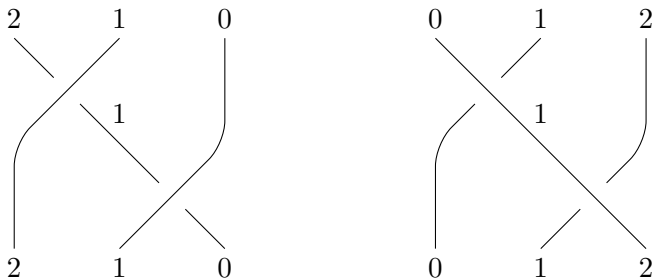

\subsection{Proofs of Theorems \ref{theorem:main1} and \ref{theorem:main2}}

The proofs of Theorems \ref{theorem:main1} and \ref{theorem:main2} are pieced together from the previous three subsections.
\begin{proof}[Proof of Theorem \ref{theorem:main1}]
Proposition \ref{prop:Whitehead} shows that $g_T(W_n) - \alt(W_n)\geq n-1$ and Inequality \ref{ineq:TuraevDalt} shows that $\dalt(W_n)-\alt(W_n)\geq n-1$. Thus $\alt(\mathcal{F}(W_n)) \ll g_T(\mathcal{F}(W_n))$ and $\alt(\mathcal{F}(W_n)) \ll \dalt(\mathcal{F}(W_n))$. Inequality \ref{ineq:TuraevJones} implies that $c(L)-\Span V_L(t)$ dominates the alternation number on $\mathcal{F}(W_n)$.

Fix an even integer $p\geq 4$ and let $q$ range over all positive integers such that $p$ and $q$ are relatively prime, and $q\neq 1 \mod p$. Proposition \ref{prop:AltGenModTor} states that $g_{\alt}(\widetilde{T}(p,q))=1$ while Corollary \ref{cor:modTor} implies that $\alt(\widetilde{T}(p,q))$ and $g_T(\widetilde{T}(p,q))$ go to infinity as $q$ goes to infinity. Hence $g_{\alt}(\mathcal{F}(\widetilde{T}(p,q))) \ll \alt(\mathcal{F}(\widetilde{T}(p,q)))$ and $g_{\alt}(\mathcal{F}(\widetilde{T}(p,q))) \ll g_T(\mathcal{F}(\widetilde{T}(p,q)))$. Both Inequality \ref{eq::altdalt} and Inequality \ref{ineq:TuraevDalt} imply that $g_{\alt}(\mathcal{F}(\widetilde{T}(p,q))) \ll \dalt(\mathcal{F}(\widetilde{T}(p,q)))$. Inequality \ref{ineq:TuraevJones} implies that the difference $c(L) - \Span V_L(t)$ dominates the alternating genus on $\mathcal{F}(\widetilde{T}(p,q))$.

Now let $q$ be a positive integer not divisible by $3$ with $q>3$.  Propositions \ref{prop:AltTor3}, \ref{prop:TuraevDaltTor3}, \ref{prop:SpanTor3}, and \ref{prop:WarpTor3} yield
\begin{align*}
c(T(3,q)) - \Span V_{T(3,q)}(t)& =  q-1,\\
\dalt (T(3,q)) &=  \lfloor q/3 \rfloor,\\
g_T (T(3,q))& =  \lfloor q/3 \rfloor,\\
\alt (T(3,q)) & =  \lfloor q/3 \rfloor~\text{or}~  \lfloor q/3 \rfloor -1,~\text{and}\\
\warp(T(3,q))& =  1/2,
\end{align*}
which implies statements (3) and (4).
\end{proof}

\begin{proof}[Proof of Theorem \ref{theorem:main2}]
Proposition \ref{prop:Whitehead} implies that for the positive $n$-th iterated untwisted Whitehead double $W_n$ of the figure-eight knot, $\alt(W_n)<g_{\alt}(W_n)$ for each positive integer $n$. Furthermore, we saw in the proof of Theorem \ref{theorem:main1} that $g_{\alt}(\mathcal{F}(\widetilde{T}(p,q))) \ll\alt(\mathcal{F}(\widetilde{T}(p,q)))$. Therefore the alternation number and alternating genus of a link are not comparable.
\end{proof}

\section{Further Questions}
\label{section:questions}

Abe \cite{Abe:Dealternating} showed that $g_T(D)\leq \dalt(D)$ for any link diagram $D$. For many diagrams, this inequality is strict, however the following question remains open.
\begin{question}
\label{q1}
Is there a link $L$ such that $g_T(L)<\dalt(L)$?
\end{question}
We now describe a method to construct diagrams where $\dalt(D)-g_T(D)$ is arbitrarily large. Let $x$ be a crossing in a link diagram $D$ considered as a $2$-tangle. An alternating $2$-tangle $\tau$ is said to {\em extend} the crossing $x$, if there is some choice of resolutions of all but one crossing of $\tau$ such that the resulting tangle is isotopic to $x$ through a planar isotopy fixing the endpoints of the tangle. If $D$ is a diagram with crossing $x$ that is extended by the rational tangle $\tau$, then let $D(x,\tau)$ be the diagram obtained by replacing the $2$-tangle $x$ with $\tau$.
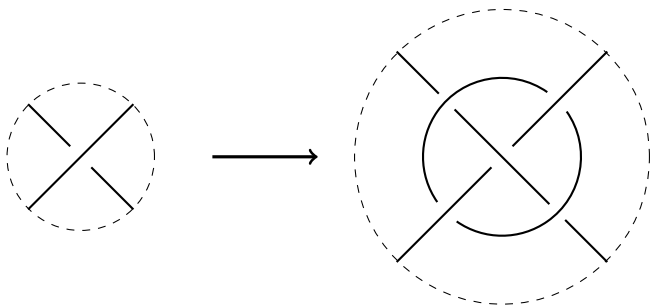
\begin{figure}[h]
$$\begin{tikzpicture}[scale=.7]

\begin{scope}[xshift= -8cm]
\draw [thick] (2,2) -- (4,4);
\draw [thick] (2,4) -- (2.8,3.2);
\draw [thick] (3.2,2.8) -- (4,2);
\draw [dashed] (3,3) circle (1.4cm);

\draw[very thick, ->] (5.5,3) -- (7.5,3);

\end{scope}

\draw[dashed] (3,3) circle (2.8cm);
\draw[thick] (1,1) -- (2.8,2.8);
\draw[thick] (3.2,3.2) -- (5,5);
\draw[thick] (1,5) -- (1.8,4.2);
\draw[thick] (2.1,3.9) -- (3.9,2.1);
\draw[thick] (4.2,1.8) -- (5,1);

\begin{scope}[xshift=3cm, yshift=3cm]

\draw[thick] (55:1.5cm) arc (55:215:1.5cm);
\draw[thick] (35:1.5cm) arc (35: -125:1.5cm);


\end{scope}

\end{tikzpicture}$$
\caption{The crossing to the left is extended by the alternating tangle to the right.}
\label{figure:tangle}
\end{figure}

Let $D$ be a link diagram with $c(D)$ crossings such that $\dalt(D)=k$. Suppose that $x$ is one of the $k$ crossings changed to make $D$ alternating, and suppose that $\tau$ is an alternating tangle with $c(\tau)$ crossings extending $x$. In order to minimally transform $D(x,\tau)$ into an alternating diagram, one must either change the $k-1$ other crossing from above along with every crossing of $\tau$ or one must change every crossing other than the $k$ crossings from above. Hence
$$\dalt(D(x,\tau)) =\min\{c(D)-k, c(\tau)+k-1\}.$$
One can show that $|s_A(D)| + |s_B(D)| - c(D) = |s_A(D(x,\tau))| + |s_B(D(x,\tau))| - c(D(x,\tau))$, and therefore $g_T(D(x,\tau))=g_T(D)$. Hence a suitably chosen sequence of alternating tangle extensions can force the gap between the dealternating number of a diagram and genus of the Turaev surface of the diagram to grow without bound. 

Question \ref{q1} being open means that it is not yet determined whether the dealternating number and the Turaev genus are actually different invariants. It is natural to ask whether Inequality \ref{ineq:TuraevJones} has an analog for the dealternating number.
\begin{question}
\label{q1a}
Is $\dalt(L) \leq c(L) - \Span V_L(t)$ for any link $L$. 
\end{question}
Turaev \cite{Turaev:SimpleProof} shows that $g_T(D)\leq c(D) - \Span V_L(t)$ for any diagram $D$ of a link $L$. However, this approach does not work for the dealternating number because there exists a diagram where $\dalt(D) > c(D) - \Span V_L(t)$. Let $D$ be the usual diagram $(5,-3,2)$ pretzel knot, and let $K(5,-3,2)$ be the knot with diagram $D$. Then $K(5,-3,2)$ is knot $10_{125}$ in the Rolfsen table. We have that $\dalt(D)=3$ while $c(D)=10$ and $\Span V_{K(5,-3,2)}(t) = 8$. Thus $\dalt(D) = 3 > 2 = c(D) - \Span V_{K(5,-3,2)}(t)$. Kim and Lee \cite{KimLee:Pretzel} prove that every non-alternating pretzel link has dealternating number one, and hence $K(5,-3,2)$ has a some diagram (not $D$) implying that $\dalt(K(5,-3,2))=1$.

By examining the behavior of Whitehead doubles, we saw that the Turaev genus dominates the alternation number of a link. However, the following question remains open.
\begin{question}
\label{q2}
Are the alternation number and the Turaev genus of a link comparable?
\end{question}
If the alternation number and Turaev genus of a link are comparable, then we must have $\alt(L)\leq g_T(L)$ for any link $L$. If they are not comparable, then there exists a link $L$ such that $g_T(L) < \alt(L)$, and hence Inequality \ref{eq::altdalt} would imply an affirmative answer to Question \ref{q1}.

Inequality \ref{ineq:warpdalt} follows from the fact that $\warp(D)\leq \dalt(D)$ for every link diagram $D$. Since Turaev genus dominates the warping span, the following question is natural. 
\begin{question}
\label{q3}
Are the warping span and the Turaev genus of a link comparable?
\end{question}
If warping span and Turaev genus are comparable, then $\warp(L)\leq g_T(L)$ for any link $L$. However, this cannot be proved in the same way as Inequality \ref{ineq:warpdalt} because there exist diagrams $D$ with $\warp(D)>g_T(D)$. For example, let $D$ be the diagram of the unknot obtained by taking the closure of the $2$-braid $\sigma_1^{-1}\sigma_1\sigma_1^{-1}\sigma_1\sigma_1^{-1}$. A straightforward computation shows that $\warp(D)=2$ while $g_T(D)=1$. 

If the warping span and Turaev genus are not comparable, then there exists a link $L$ with $g_T(L) < \warp(L)$. In order to confirm such an example, it would be useful to find lower bounds for the warping span of a link. The case of alternating genus is similar. Proposition \ref{prop:hyperbolic} gives an obstruction for certain links to have alternating genus one, but beyond that there is no useful lower bound for the alternating genus of a link. This leads naturally to our final question.
\begin{question}
Do links with arbitrarily large alternating genus or warping span exist?
\end{question}

\bibliography{linklit1}{}
\bibliographystyle {amsalpha}

\end{document}